\documentclass{cedram-afst-ppn}
\usepackage[english,frenchb]{babel}
\usepackage[latin1]{inputenc}
\usepackage[T1]{fontenc}
\usepackage{lmodern}
\usepackage{graphicx}
\usepackage{epstopdf}
\usepackage{color}
\usepackage[colorlinks=true]{hyperref}
\hypersetup{urlcolor=blue, citecolor=red}
\usepackage{amsfonts,amsmath,amsthm,amssymb,amsxtra}

\newcommand{\R}{{\mathbb R}}
\renewcommand{\S}{{\mathbb S}}

\newcommand{\be}[1]{\begin{equation}\label{#1}}
\newcommand{\ee}{\end{equation}}
\renewcommand{\(}{\left(}
\renewcommand{\)}{\right)}
\renewcommand{\S}{\mathbb{S}}
\newcommand{\iS}[1]{\int_{\S^d}{#1}\;d\mu}
\newcommand{\iStwo}[1]{\int_{\S^2}{#1}\;d\mu}
\newcommand{\nrm}[2]{\|{#1}\|_{\mathrm L^{#2}(\S^d)}}
\newcommand{\Lap}{\Delta_{\S^d}}

\renewcommand{\L}{{\mathcal L}\,}
\newcommand{\ix}[1]{\int_{-1}^1{#1}\;d\nu_d}
\newcommand{\nrmx}[2]{\|#1\|_{#2}}
\newcommand{\scal}[2]{\left\langle{#1},{#2}\right\rangle}
\def\theoname{Theorem}
\def\lemmname{Lemma}
\def\keywordsname{Keywords}

\title[Interpolation inequalities on the sphere]{Interpolation inequalities on the sphere:\\ linear \emph{vs.} nonlinear flows}
\alttitle{Inégalités d'interpolation sur la sphère:\\ flots non-linéaires \emph{vs.} flots linéaires}
\author[J.~Dolbeault]{\firstname{Jean} \lastname{Dolbeault}}
\address{Ceremade (UMR CNRS 7534)\\ Université Paris-Dauphine\\ Place de Lattre de Tassigny\\ 75775 Paris Cédex~16\\ France}
\email{dolbeaul@ceremade.dauphine.fr}
\author[M.J.~Esteban]{\firstname{Maria J.} \lastname{Esteban}}
\address{Ceremade (UMR CNRS 7534)\\ Université Paris-Dauphine\\ Place de Lattre de Tassigny\\ 75775 Paris Cédex~16\\ France}
\email{esteban@ceremade.dauphine.fr}
\author[M.~Loss]{\firstname{Michael} \lastname{Loss}}
\address{School of Mathematics,\\ Skiles Building\\ Georgia Institute of Technology\\ Atlanta GA 30332-0160\\ USA}
\email{loss@math.gatech.edu}
\thanks{Partially supported by the projects \emph{STAB} and \emph{Kibord} (J.D.) of the French National Research Agency (ANR), and by the NSF grant DMS-1301555 (M.L.)}

\altkeywords{Interpolation; inégalités fonctionnelles; flots; constantes optimales; équations elliptiques semi-linéaires; rigidité; unicité; méthode du \emph{carré du champ}; condition CD($\rho$,$N$); équation de la chaleur; diffusion non-linéaire; inégalité de trou spectral; inégalité de Poincaré; inégalités améliorées}

\begin{document}

\selectlanguage{english}\begin{abstract} This paper is devoted to sharp interpolation inequalities on the sphere and their proof using flows. The method explains some rigidity results and proves uniqueness in related semilinear elliptic equations. Nonlinear flows allow to cover the interval of exponents ranging from Poincaré to Sobolev inequality, while an intriguing limitation (an upper bound on the exponent) appears in the \emph{carré du champ} method based on the heat flow. We investigate this limitation, describe a counter-example for exponents which are above the bound, and obtain improvements below.\end{abstract}

\selectlanguage{french}\begin{altabstract} Cet article est consacré à des inégalités d'interpolation optimales sur la sphère et à leur preuve par des flots. La méthode explique aussi certains résultats de rigidité et permet de prouver l'unicité dans des équations elliptiques semilinéaires associées. Les flots non-linéaires permettent de couvrir tout l'intervalle des exposants entre l'inégalité de Poincaré et l'inégalité de Sobolev, tandis qu'une limitation intrigante (une limite supérieure de l'exposant) apparaît dans la méthode du carré du champ basée sur le flot de la chaleur. Nous étudions cette limitation, décrivons un contre-exemple pour les exposants qui sont au-dessus de la borne, et obtenons des améliorations en-dessous.\end{altabstract}

\maketitle
\thispagestyle{empty}
\selectlanguage{english}

\section{Introduction}\label{Sec:Intro}

On the $d$-dimensional sphere, let us consider the interpolation inequality
\be{Ineq:GNS}
\nrm{\nabla u}2^2+\frac d{p-2}\,\nrm u2^2\ge\frac d{p-2}\,\nrm up^2\quad\forall\,u\in\mathrm H^1(\S^d,d\mu)\,,
\ee
where the measure $d\mu$ is the uniform probability measure on $\S^d\subset\R^{d+1}$ corresponding to the measure induced by the Lebesgue measure on $\R^{d+1}$, and the exposant $p\ge1$, $p\neq2$, is such that
\[
p\le2^*:=\frac{2\,d}{d-2}
\]
if $d\ge3$. We adopt the convention that $2^*=\infty$ if $d=1$ or $d=2$. The case $p=2$ corresponds to the logarithmic Sobolev inequality
\be{Ineq:logSob}
\nrm{\nabla u}2^2\ge\frac d2\,\iS{|u|^2\,\log\(\frac{|u|^2}{\nrm u2^2}\)}\quad\forall\,u\in\mathrm H^1(\S^d,d\mu)\setminus\{0\}\,.
\ee
In both cases, equality is achieved by any constant non-zero function and constants are optimal. Indeed, if we define
\[
\mathcal Q_p[u]:=\frac{(p-2)\,\nrm{\nabla u}2^2}{\nrm up^2-\nrm u2^2}\quad\mbox{and}\quad\mathcal Q_2[u]:=\frac{2\,\nrm{\nabla u}2^2}{\iS{|u|^2\,\log\Big(\frac{|u|^2}{\nrm u2^2}\Big)}}
\]
respectively for $p\neq2$ and for $p=2$, and consider an eigenfunction $\varphi$ associated with the first positive eigenvalue of the Laplace-Beltrami operator on $\S^d$, optimality can be checked by computing $\mathcal Q_p[1+\varepsilon\,\varphi]$ as $\varepsilon\to0$. Inequality~\eqref{Ineq:GNS} has been established in~\cite{BV-V} by rigidity methods, in~\cite{MR1230930} by techniques of harmonic analysis, and using the \emph{carré du champ} method in~\cite{MR1231419,MR1412446,MR2381156}, for any $p>2$. The case $p=2$ was studied in~\cite{MR674060}.

Here we shall focus on flow methods. In~\cite{MR772092,Bakry-Emery85,MR808640}, D.~Bakry and M.~Emery proved the inequalities using the \emph{heat flow} provided
\[
p\le2^\#:=\frac{2\,d^2+1}{(d-1)^2}\,.
\]
This special exponent is emphasized in~\cite{MR808640}. It is an important limitation, as we shall see in Section~\ref{Sec:Counter-Example}. Up to now, it was not known whether the limitation was of technical nature, or if there was a deep reason for it. Our main result is to build a counter-example which shows why heat flow methods definitely cannot cover the whole range of the exponents up to the critical exponent $2^*$ while nonlinear flows, with a proper choice of the nonlinearity, do it. \emph{Nonlinear flows} introduced in~\cite{MR2381156} provide a unified framework for \emph{rigidity} and \emph{carré du champ} methods as shown in~\cite{1302}. We refer to~\cite{MR2213477,MR3155209} for background references. More specialized papers will be quoted below.

On the other hand, in the range $p<2^\#$ which is covered by a heat flow method, we provide an \emph{improved inequality} with a constructive method under an integral constraint on the set of functions. See next section for details. We also provide a constructive estimate when $p\in[2^\#,2^*]$ under an antipodal symmetry contraint: see Theorem~\ref{Thm:Antipodal}.

The flow method applies to general compact manifolds but optimality is achieved only for spheres and not in the general case. The reader interested in differential geometry issues is invited to refer to~\cite{1302} and many other papers quoted therein. We will focus on the case of the sphere and use a simplified version of the inequality based on the ultraspherical operator to build our counter-examples.

\section{Flows and functional inequalities}\label{Sec:Flows}

If we define the functionals $\mathcal E_p$ and $\mathcal I_p$ respectively by
\[
\mathcal E_p[\rho]:=\frac1{p-2}\,\left[\iS{\rho^\frac2p}-\(\iS\rho\)^\frac2p\right]\quad\mbox{if}\quad p\neq2\,,
\]
\[
\mathcal E_2[\rho]:=\iS{\rho\,\log\(\frac\rho{\nrm\rho1}\)}\,,
\]
for $\rho>0$, and
\[
\mathcal I_p[\rho]:=\iS{|\nabla\rho^\frac1p|^2}\,,
\]
then inequalities~\eqref{Ineq:GNS} and~\eqref{Ineq:logSob} amount to $\mathcal I_p[\rho]\ge d\,\mathcal E_p[\rho]$ as can easily be checked using $\rho=|u|^p$. To establish such inequalities, one can use the heat flow
\be{Eqn:Heat}
\frac{\partial\rho}{\partial t}=\Delta\rho
\ee
where $\Delta$ denotes the Laplace-Beltrami operator on $\S^d$, and compute
\[
\frac d{dt}\mathcal E_p[\rho]=-\,\mathcal I_p[\rho]\quad\mbox{and}\quad\frac d{dt}\mathcal I_p[\rho]\le-\,d\,\mathcal I_p[\rho]\,.
\]
Details of the computation based on the \emph{carré du champ} will be given below. However, there is a strict limitation on the exponent, namely that $p\le2^\#$. If this condition is satisfied, we obtain that
\[
\frac d{dt}\Big(\mathcal I_p[\rho]-\,d\,\mathcal E_p[\rho]\Big)\le0\,.
\]
On the other hand, $\rho(t,\cdot)$ converges as $t\to\infty$ to a constant, namely $\iS\rho$ since $d\mu$ is a probability measure and $\iS\rho$ is conserved by~\eqref{Eqn:Heat}. As a consequence, $\lim_{t\to\infty}\(\mathcal I_p[\rho]-d\,\mathcal E_p[\rho]\)=0$, which proves that $\mathcal I_p[\rho(t,\cdot)]$ $ -\,d\,\mathcal E_p[\rho(t,\cdot)]\ge0$ for any $t\ge0$ and completes the proof. See~\cite{MR808640} for details. One may wonder whether the monotonicity property is also true for some $p>2^\#$. Our first result contains a negative answer to this question.
\begin{prop}\label{Prop:Counter-Example} For any $p\in(2^\#,2^*)$ or $p=2^*$ if $d\ge3$, there exists a function $\rho_0$ such that, if $\rho$ is a solution of~\eqref{Eqn:Heat} with initial datum $\rho_0$, then
\[
\frac d{dt}\Big(\mathcal I_p[\rho]-\,d\,\mathcal E_p[\rho]\Big)_{|t=0}>0\,.
\]
\end{prop}
To overcome the limitation $p\le2^\#$, one can consider a nonlinear diffusion of fast diffusion / porous medium type
\be{Eqn:FDE}
\frac{\partial\rho}{\partial t}=\Delta\rho^m\,.
\ee
With this flow, we no longer have $\frac d{dt}\mathcal E_p[\rho]=-\,\mathcal I_p[\rho]$ but can still prove that
\[
\mathcal K_p[\rho]:=\frac d{dt}\Big(\mathcal I_p[\rho]-\,d\,\mathcal E_p[\rho]\Big)\le0\,,
\]
for any $p\in[1,2^*]$. Proofs have been given in~\cite{MR2381156,1302}. We also refer to~\cite{DEKL2012,DEKL2014} for results which are more specific to the case of the sphere and of the ultraspherical operator, and further references therein. Except for $p=1$ and $p=2^*$ with $d\ge3$, there is some flexibility in the choice of $m$. It is enough to pick a special example for proving Proposition~\ref{Prop:Counter-Example}. Notice that we use a function related with the nonlinear diffusion equation~\eqref{Eqn:FDE} to prove the non-monotonicity property along the heat flow~\eqref{Eqn:Heat}. See Section~\ref{Sec:Counter-Example} for details and for a proof of Proposition~\ref{Prop:Counter-Example}.

For any $p<2^*$, existence of optimal functions in~\eqref{Ineq:GNS} and~\eqref{Ineq:logSob} is not an issue due to the compactness of Sobolev's embeddings. Instead of considering the whole flow, it is possible to take such an optimal function~$u$ (or more generically a positive critical point) as initial datum, compute the time-derivative $\mathcal K_p$ using the flow at $t=0$ (which is equal to $0$ because~$u$ is a critical point of $\mathcal I_p-\,d\,\mathcal E_p$), and use this computation to identify~$u$. This is the essence of the rigidity method as in~\cite{BV-V,MR1412446}: see~\cite{1302} for details and improvements. In the flow perspective, we can also make use of $\mathcal K_p$ to obtain improved inequalities: see~\cite{DEKL2014}. Here we use a function $u$ such that $\mathcal K_p[u]=0$ (along the nonlinear flow~\eqref{Eqn:FDE}) as initial datum for~\eqref{Eqn:Heat}, when $p=2^*$, and check that, for an appropriate choice of $m$, it satisfies the property of Proposition~\ref{Prop:Counter-Example}.

\medskip With no restriction, we can assume that $\iS\rho=1$. As $t\to\infty$, the equation~\eqref{Eqn:FDE} becomes equivalent to the heat flow~\eqref{Eqn:Heat}, which allows to relate best constants in~\eqref{Ineq:GNS} and~\eqref{Ineq:logSob} with the spectral gap, or Poincar\'e inequality, associated with the Laplace-Beltrami operator. Because of the improved inequalities that have been shown in~\cite{DEKL2014} (see the proof of Proposition~\ref{Prop:BE}), optimality can be achieved only in the asymptotic regime. This explains why the computation of $\mathcal Q_p[1+\varepsilon\,\varphi]$ as $\varepsilon\to0$ mentioned in Section~\ref{Sec:Intro} provides the optimal constant if $\varphi$ is an eigenfunction associated with the first positive eigenvalue of the Laplace-Beltrami operator. This also raises a very interesting question that we address in Section~\ref{Sec:Improvements} and goes as follows. If we assume that the initial datum satisfies
\[
\iS{x\,\rho}=0\,,
\]
is the decay rate of $\mathcal I_p-\,d\,\mathcal E_p$ along~\eqref{Eqn:FDE} faster and can we write that $\frac d{dt}\(\mathcal I_p[\rho]-\,\lambda\,\mathcal E_p[\rho]\)\le0$ for some $\lambda>d$ ? In other words, can we improve on the value of the infimum of $\mathcal Q_p[u]$ if we assume that $\iS{x\,|u|^p}=0$~? Notice indeed that, in the asymptotic regime as $t\to\infty$, this condition means that the solution of~\eqref{Eqn:Heat} is orthogonal to the eigenspace corresponding to the first positive eigenvalue of $-\,\Delta$, and hence proves that, for any $\varepsilon>0$, there exists a constant $C>0$, depending on $\varepsilon$, such that
\[
\mathcal E_p[\rho(t,\cdot)]\le C\,e^{-\big(2\,(d+1)-\varepsilon\big)\,t}\quad\forall\,t\ge0\,.
\]

Some partial results are known.

\smallskip
$\bullet$ If $p=1$, that is, in the linear case, Inequality~\eqref{Ineq:GNS} is equivalent to a Poincar\'e inequality
\[
\nrm{\nabla u}2^2\ge d\,\nrm{u-1}2^2\quad\forall\,u\in\mathrm H^1(\S^d,d\mu)\;\mbox{ such that}\;\iS u=1\,.
\]
With the additional condition that $\iS{x\,u}=0$, the inequality is improved to
\[
\nrm{\nabla u}2^2\ge2\,(d+1)\,\nrm{u-1}2^2
\]
as can be shown by a simple decomposition in spherical harmonics.
\\
$\bullet$ If $p=2^*$ and $d\ge3$, G.~Bianchi and H.~Egnell have shown in~\cite{MR1124290} that the Euclidean Sobolev inequality can be improved. Using an inverse stereographic projection, this exactly shows that $\lambda>d$ and we will give a similar argument in Section~\ref{Sec:Improvements}. However, this is argued by contradiction so that no explicit value of $\lambda$ is given.
\\
$\bullet$ If $d=2$, then $2^*=\infty$ and the Sobolev inequality has to be replaced by the Moser-Trudinger-Onofri inequality: see~\cite{1401} for considerations in this direction. This inequality states that
\begin{multline*}
\log\(\iStwo{e^u}\)\le\frac\alpha4\nrm{\nabla u}2^2\quad\forall\,u\in\mathrm H^1(\S^2,d\mu)\\
\mbox{ such that}\;\iStwo u=0\,,
\end{multline*}
with $\alpha=1$. It has been conjectured by A.~Chang and P.~Yang that $\alpha=1/2$ under the additional condition that $\iStwo{x\,e^u}=0$, but so far the best existing result has been obtained in~\cite{MR2670931} and shows that $\alpha\le2/3$.

Of course, a major difficulty comes from the fact that the property $\iS{x\,\rho}=0$ is not conserved by the flow of~\eqref{Eqn:FDE}, except if $m=1$ (and~\eqref{Eqn:FDE} coincides then with~\eqref{Eqn:Heat}), as we shall see next. This is why we can produce an explicit estimate for $\lambda$ only in the range $p\le2^\#$. Let us define
\[
\Lambda^\star\quad:=\inf_{\begin{array}{c}v\in\mathrm H^1_+(\S^d,d\mu)\\ \iS{v}=1\\ \iS{x\,|v|^p}=0\end{array}}\frac{\iS{{|\nabla v|^2}}}{\iS{|v-1|^2}}\,.
\]
Here $\mathrm H^1_+(\S^d,d\mu)$ denotes the a.e.~nonnegative functions in $\mathrm H^1(\S^d,d\mu)$.
\begin{thm}\label{Thm:Improvement} For any $p\in(2,2^*)$, there exists a constant $\Lambda>d$ such that
\be{Ineq:GNSimproved}
\nrm{\nabla u}2^2+\frac\Lambda{p-2}\,\nrm u2^2\ge\frac\Lambda{p-2}\,\nrm up^2
\ee
for any function $u\in\mathrm H^1(\S^d,d\mu)$ such that $\iS{x_i\,|u|^p}=0$ with $i=1,\,2,\ldots,d$. Moreover, if $p\le2^\#$, with $\Lambda^\star>d$, we have the estimate
\be{Ineq:lambda}
\Lambda\ge d+\frac{(d-1)^2}{d\,(d+2)}\,(2^\#-p)\,(\Lambda^\star-d)\,.
\ee
\end{thm}
The strategy of the proof of this result will be given in Section~\ref{Sec:Improvements}. We will also give an estimate of $\Lambda^\star$ for the limit case $p=2$ of the logarithmic Sobolev inequality in Proposition~\ref{ImprovedLogSob}.

\section{The ultraspherical operator}\label{Sec:ultraspherical}

To avoid technicalities, we will work with the ultraspherical operator instead of the Laplace-Beltrami operator. As in~\cite{DEKL2012,DEKL2014}, we can indeed take coordinates for $x=(x',z)\in\S^d$ such that $|x'|^2+z^2=1$, with $x'\in\R^d$ and $z\in[-1,1]$. A simple symmetrization argument (see for instance~\cite{MR3052352}) shows that optimality in~\eqref{Ineq:GNS} and~\eqref{Ineq:logSob} is achieved by functions depending only on $z$ so that, in order to prove these inequalities, it is equivalent to establish the inequalities
\be{GNS-ultra}
-\scal f{\L f}=\ix{|f'|^2\;\nu}\ge\frac d{p-2}\,\(\nrmx fp^2-\nrmx f2^2\)
\ee
and
\be{LogSob-ultra}
-\scal f{\L f}\ge\frac d2\ix{|f|^2\,\log\(\frac{|f|^2}{\nrmx f2^2}\)}\,,
\ee
for any function $f\in\mathrm H^1((-1,1),d\nu_d)$. Here $\nrmx fq:=\(\ix{|f|^q}\)^{1/q}$ and $\L f$ denotes the \emph{ultraspherical} operator given by
\[
\L f:=(1-z^2)\,f''-d\,z\,f'=\nu\,f''+\frac d2\,\nu'\,f'\,,
\]
while $d\nu_d$ is the probability measure defined by
\[
\nu_d(z)\,dx=d\nu_d(z):=Z_d^{-1}\,\nu^{\frac d2-1}\,dx\quad\mbox{with}\quad\nu(z):=1-z^2\,,
\]
and the normalization constant is $Z_d=\sqrt\pi\,\frac{\Gamma\big(\tfrac d2\big)}{\Gamma\big(\tfrac{d+1}2\big)}$.

With the scalar product $\scal{f_1}{f_2}=\ix{f_1\,f_2}$ defined on the space $\mathrm L^2((-1,1),d\nu_d)$, let us recall that the main property of $\L$ is
\[
\scal{f_1}{\L f_2}=-\ix{f_1'\,f_2'\;\nu}\,.
\]
We refer to~\cite{MR674060,MR1260331,MR2641798,MR1231419,MR1971589,MR1435336,MR1651226,DEKL2012} for more references. The next lemma, which is taken from~\cite[Inequalities~(3.2)~and~(3.3)]{DEKL2012}, gives two elementary but very useful identities.
\begin{lemma}\label{lem:identies}
For any positive smooth function $f$ on $(-1,1)$, we have
\begin{eqnarray*}
&&\ix{(\L f)^2}=\ix{|f''|^2\;\nu^2}+d\ix{|f'|^2\;\nu}\,,\\
&&\scal{\frac{|f'|^2}f\;\nu}{\L f}=\frac d{d+2}\kern-1.5pt\ix{\frac{|f'|^4}{f^2}\;\nu^2}-\,2\,\frac{d-1}{d+2}\kern-1.5pt\ix{\frac{|f'|^2\,f''}f\;\nu^2}\,.
\end{eqnarray*}\end{lemma}

Now let us rephrase the flow methods in the framework of the ultraspherical operator. With $\rho=|f|^p$, Inequality~\eqref{GNS-ultra} can be rewritten as
\[
\mathsf F[\rho]:=\frac1d\ix{|(\rho^{1/p})'|^2\,\nu}-\frac1{p-2}\left[\(\ix\rho\)^{2/p}-\ix{\rho^{2/p}}\right]\ge0
\]
if $p\neq2$. In the case $p=2$, Inequality~\eqref{LogSob-ultra} can be rewritten as
\[
\mathsf F[\rho]:=\frac1d\ix{|(\sqrt\rho)'|^2\,\nu}-\frac12\ix{\rho\log\(\frac\rho{\ix\rho}\)}\ge0\,.
\]
Let us consider its evolution that along the heat flow
\be{heat}
\frac{\partial\rho}{\partial t}=\L\rho\,,
\ee
The following result has been established by D.~Bakry and M.~Emery in~\cite{MR808640}.
\begin{prop}\label{Prop:linear} Assume that either $d>1$ and $p\in[1,2^\sharp]$, or $d=1$ and $p\ge1$. If $\rho$ solves~\eqref{heat}, then the functional $\mathsf F[\rho]$ is nonincreasing.\end{prop}

But if we consider $p$ belonging to the larger interval $[1,2^*]$, the functional $\mathsf F[\rho]$ is nonincreasing along the fast diffusion / porous medium flow
\be{FDE}
\frac{\partial\rho}{\partial t}=\L\rho^m\,,
\ee
with
\be{m}
m=1+\frac2p\(\frac1\beta-1\)\,,
\ee
and an appropriate choice of $\beta$: see~\cite{MR2381156,1302,DEKL2012,DEKL2014} for detailed results. Here is a summary of the results when $p\ge 1$. Let us define the numbers
\[
\beta_\pm(p,d):=\frac{d^2-d\,(p-5)-2\,p+6\pm(d+2)\,\sqrt{d\,(d-2)\,(p-1)\,(2^*-p)}}{d^2\,\(p^2-3\,p+3\)-2\,d\,(p^2-3)+(p-3)^2}\,,
\]
which are the roots of a second order polynomial $\beta\mapsto\gamma(\beta)$ whose expression can be found in Section~\ref{Sec:Counter-Example}. Notice that $\beta_+$ and $\beta_-$ coincide when $p=2$: $\beta_\pm(2^*)=(d-2)/(d-3)$. The denominator
\[
\delta(p,d):=d^2\,\(p^2-3\,p+3\)-2\,d\,(p^2-3)+(p-3)^2
\]
is positive if and only if one of the following condition is satisfied:
\begin{itemize}
\item[$\bullet$] $d\ge5$,
\item[$\bullet$] $d=4$ and $p\neq3$,
\item[$\bullet$] $d=2$ or $d=3$ and $p\not\in[p_-(d),p_+(d)]$ where $p_\pm(d)$ are the two roots of the equation $\delta(p,d)=0$,
\item[$\bullet$] $d=1$ and $p<2$.
\end{itemize}
Notice that the case $d=3$ and $p=6$ formally corresponds to $\beta=+\infty$ and deserves a spacial treatment. It is covered with $m=2/3$ in~\eqref{FDE}.
\begin{prop}\label{prop:SufficientConsidtion} Let $p\in[1,2^*]$ and either $\beta\in[\beta_-(p,d),\beta_+(p,d)]$ if $\delta(p,d)>0$, or $\beta\in(-\infty,\beta_+(p,d)]\cup[\beta_-(p,d),+\infty)$ if $\delta(p,d)<0$. If $\delta(\bar p,d)=0$ for some $\bar p=0$, we assume that the range of admissible values for $\beta$ is the limit of the range as $p\to\bar p_-$. Then $\frac d{dt}\mathsf F[\rho]\le0$ if $\rho$ solves~\eqref{FDE}.\end{prop}
The result of Proposition~\ref{Prop:linear} is obtained by checking for which values of $p$ the case $\beta=1$ is admissible in Proposition~\ref{prop:SufficientConsidtion}. In both cases, the method provides only a sufficient condition. See Figs.~\ref{F1} and~\ref{F2} for an illustration when $d=5$.
\begin{figure}[ht]
\includegraphics[width=6cm]{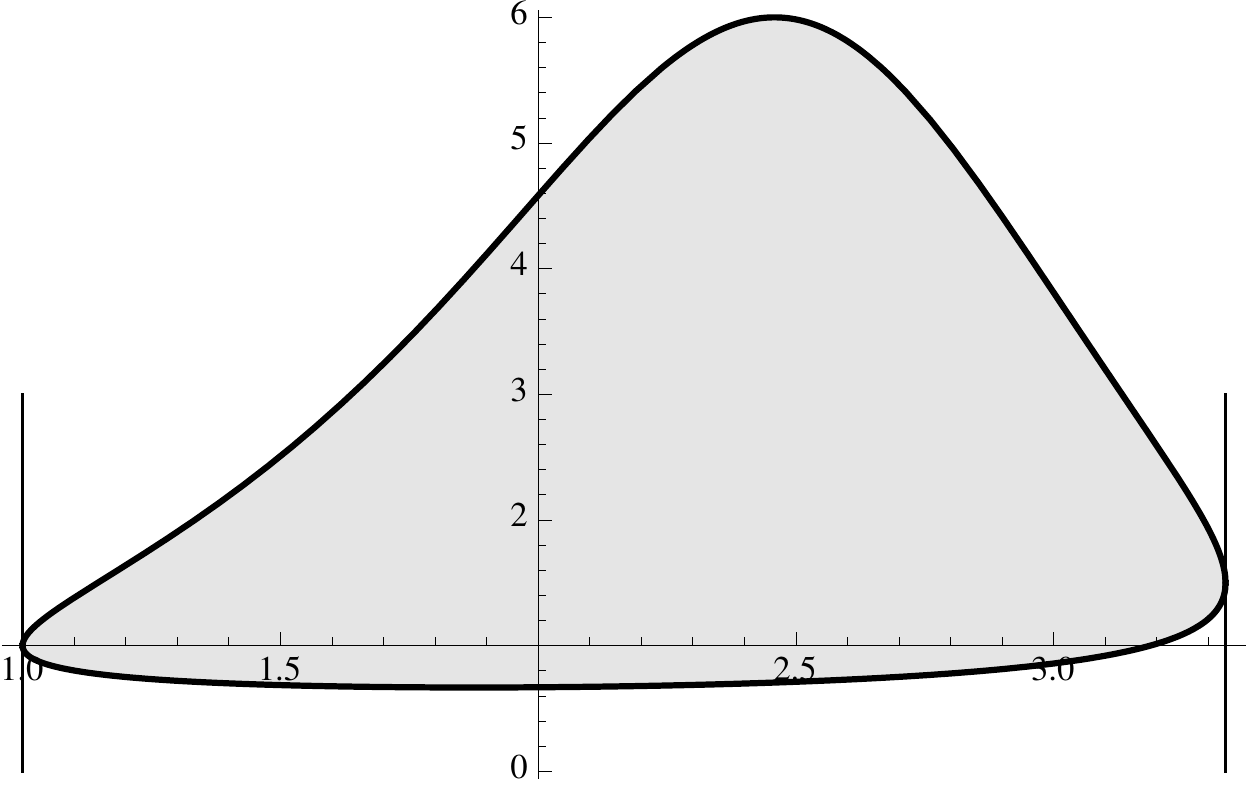}
\caption{\label{F1}\small\sl The gray area corresponds to the $(p,\beta)$ admissible region in which $\mathsf F[\rho]$ is monotone nonincreasing if~$\rho$ solves~\eqref{FDE}, in the case $d=5$. It is delimited by the curves $p\mapsto\beta_\pm(p,d)$. Similar patterns occur in higher dimensions. When $1\le d\le4$, the admissible region is slightly more complicated: see~\cite{DEKL2014} for details. In any dimension $d\ge2$, the line $\beta=1$ intersects $p\mapsto\beta_-(p)$ at $p=2^\#$. For any $d\ge1$, there exists an admissible value of $\beta$ for any $p\in[1,2^*)$ and also for $p=2^*$ if $d\ge3$.}
\end{figure}
\begin{figure}[ht]
\includegraphics[width=6cm]{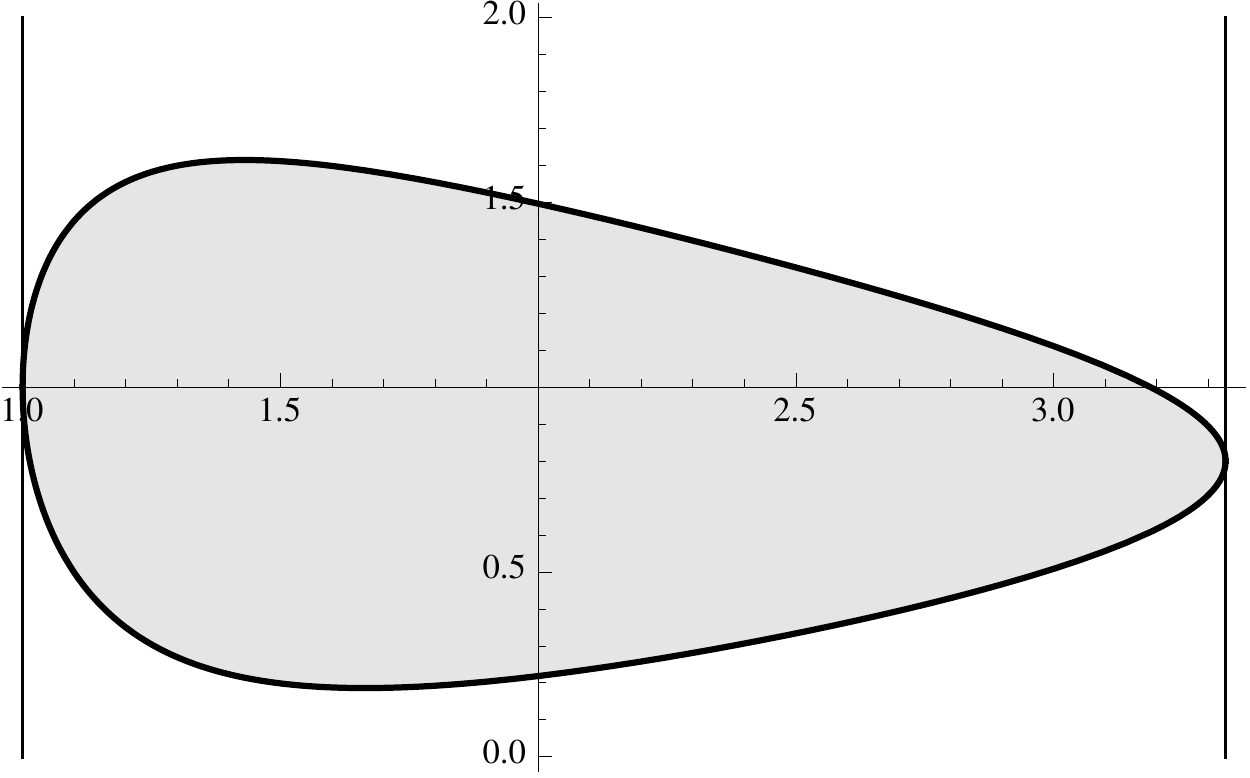}
\caption{\label{F2}\small\sl The gray area corresponds to the $(p,m)$ admissible region, in the case $d=5$, where $m$ is the exponent in~\eqref{FDE} and given in terms of $m$ by~\eqref{m}. The case $m>1$ and $m<1$ correspond respectively to the porous medium and fast diffusion cases, while the threshold case $m=1$, which is limited to $p\in[1,2^\#]$, is the special case of the heat equation~\eqref{heat}.}
\end{figure}

\section{A counter-example for the heat flow}\label{Sec:Counter-Example}

The conditions $p\le2^\#$ and $p\le2^*$ are only sufficient conditions for the monotonicity of $\mathsf F[\rho]$ under the action of~\eqref{heat} and~\eqref{FDE}, and one can wonder, for instance, if the monotonicity can be established for larger values of $p$ under the action of the heat flow~\eqref{heat}.

\medskip A {\bf\emph{first obstruction}} arises from the fact that for $p=2^*$, due to conformal invariance properties on the sphere, optimality in~\eqref{GNS-ultra} is achieved not only by the constant functions but also by any function of the form
\be{ConfSolutions}
u(z)=\(\mathsf a+\mathsf b\,z\)^{-\frac{d-2}2}\quad\forall\,z\in(-1,1)\,.
\ee
 Indeed we have the following technical result.
\begin{prop}\label{prop:obstruction1} If $d\ge3$ and $p=2^*$, the function
\[
\rho(t,x)=\big(\mathsf a(t)+\mathsf b(t)\,z\big)^{-d}
\]
is positive and solves~\eqref{FDE} with $m=1-\frac1d$ if and only if
\begin{eqnarray*}
&&\mathsf a(t)=\omega\,\coth\((d-1)\,\omega\,(t+t_0)\)\,,\\
&&\mathsf b(t)=\pm\,\omega\,\mathrm{csch}\((d-1)\,\omega\,(t+t_0)\)\,,
\end{eqnarray*}
for some nonnegative integration constants $\omega$ and $t_0$.\end{prop}
\begin{proof} Inserting the expression of $\rho$ in~\eqref{FDE}, we get that
\[
\mathsf a'+\mathsf b'\,z=-\,\mathsf b\,(d-1)\,(\mathsf b+\mathsf a\,z)
\]
for any $z\in(-1,1)$. Hence $(\mathsf a,\mathsf b)$ solve the system
\[
\mathsf a'=-\,(d-1)\,\mathsf b^2\,,\quad\mathsf b'=-\,(d-1)\,\mathsf a\,\mathsf b\,.
\]
{}From the positivity of $\rho$, we deduce that $\mathsf a>|\mathsf b|$ and deduce that
\[
\frac{\mathsf a'}{\mathsf b'}=\frac{\mathsf b}{\mathsf a}\,.
\]
There exists a positive constant $\omega$ such that
\[
\mathsf a=\sqrt{\omega^2+\mathsf b^2}
\]
and the problem is reduced to
\[
\mathsf a'=(d-1)\(\omega^2-\mathsf a^2\)\,.
\]
We conclude after integrating the ODE for $\mathsf a$ and using $\mathsf b=\pm\sqrt{\mathsf a^2-\omega^2}$.\end{proof}

As we shall see next, if $p=2^*$ and $d>3$, and~$\rho$ is a solution of~\eqref{FDE}, the only possible choice for $\beta$ compatible with Proposition~\ref{prop:SufficientConsidtion} is $\beta=\beta_\pm(2^*)=(d-2)/(d-3)$, and in this case
\be{alpha}
\frac d{dt}\mathsf F[\rho]=-\,2\,\beta^2\ix{\left|w''-\frac{d-1}{d-3}\,\frac{|w'|^2}w\right|^2\,\nu^2}\,,
\ee
with $\rho=u^p= w^{\beta p}$. When $u$ is given by~\eqref{ConfSolutions}, $w$ satisfies
\be{ODEw}
w''-\frac{d-1}{d-3}\,\frac{|w'|^2}w=0\,.
\ee
That is, for $\rho=u^p= w^{\beta p}$, solution of~\eqref{FDE}, and $u$ given by~\eqref{ConfSolutions} as initial datum, we obtain
\[
\frac d{dt}\mathsf F[\rho]_{|t=0}=0\,.
\]
If $\rho$ ($=u^p$) solves~\eqref{heat} instead of~\eqref{FDE}, we also find that $\frac d{dt}\mathsf F[\rho]_{|t=0}=0$, because $\rho$ is a minimizer of $\mathsf F[\rho]$ at $t=0$. However, it is simple to check that the family~\eqref{ConfSolutions} is not invariant under the action of~\eqref{heat}, as
\[
\frac{\partial\rho}{\partial t}=d\,(d-1)\,\mathsf b\,\frac{\mathsf b+\mathsf a\,z}{(\mathsf a+\mathsf b\,z)^{d+1}}
\]
clearly differs from
\[
\L\rho=d\,\mathsf b\,\,\frac{-\,\mathsf b\,z^2+d\,\mathsf a\,z+(d+1)\,\mathsf b}{(\mathsf a+\mathsf b\,z)^{d+2}}\,.
\]
We claim that \emph{any positive minimizer of $\mathsf F[\rho]$ is given by~\eqref{ConfSolutions} for some $\mathsf a$ and $\mathsf b$ such that $\mathsf a>|\mathsf b|$}. Indeed by~\eqref{alpha} and using the same notations as above, a minimizer solves~\eqref{ODEw}. This ODE can be solved using elementary methods and shows that $\rho(x)=(\mathsf a+\mathsf b\,z)^{-d}$.

Altogether, this proves that $\rho\mapsto\mathsf F[\rho]$ cannot evolve monotonously along the flow of~\eqref{heat}, and proves the result of~Proposition~\ref{Prop:Counter-Example} with $\rho_0=\rho(t_0,\cdot)$, for some $t_0>0$. This first obstruction is however not fully explicit.

\medskip A {\bf\emph{second obstruction}} arises from the fact that if $p\in(2^\#,2^*)$, one can find explicit functions such that $\frac d{dt}\mathsf F[\rho]_{|t=0}>0$, with $\rho$ solving~\eqref{heat}. We shall prove the following refined version of Proposition~\ref{Prop:Counter-Example}.
\begin{prop}\label{prop:obstruction2} Assume that $d\ge3$, $p\in(2^\#,2^*)$ and $\beta=\beta_-(p,d)$. There exists an explicit, non-constant, positive function $f$ and a positive constant $\mathsf A$ such that, if~$\rho$ solves~\eqref{heat} with initial datum $\rho(t=0,\cdot)=|f|^p$, then
\[
\frac d{dt}\mathsf F[\rho]_{|t=0}=\mathsf A\ix{\frac{|f'|^4}{f^2}\,\nu^2}\,.
\]
\end{prop}
Before proving this proposition, let us recall some known results for the heat flow and for the fast diffusion equation.

\medskip\noindent$\bullet$ \emph{The heat flow approach.} Assume that $p\neq2$. If $\rho=|u|^p$ solves~\eqref{heat}, then $u$ is a solution of
\be{Eqn:Linear}
\frac{\partial u}{\partial t}=\L u+(p-1)\,\frac{|u'|^2}u
\ee
with initial datum $f$ and we notice that
\[
\frac d{dt}\ix{|u|^p}=0\,,
\]
so that $\overline u^p:=\ix{|u|^p}$ is preserved. A straightforward computation (using the definition of $\mathcal L$ and Lemma~\ref{lem:identies}) shows that
\begin{multline*}
-\,\frac12\,\frac d{dt}\ix{\(|u'|^2\,\nu+\frac d{p-2}\,\big(|u|^2-\overline u^2\big)\)}\\
=\ix{|u''|^2\,\nu^2}-2\,\frac{d-1}{d+2}\,(p-1)\ix{u''\,\frac{|u'|^2}u\,\nu^2}\\
+\frac d{d+2}\,(p-1)\ix{\frac{|u'|^4}{u^2}\,\nu^2}\,.
\end{multline*}
The r.h.s.~is positive if
\[
\gamma_1=\frac d{d+2}\,(p-1)-\(\frac{d-1}{d+2}\,(p-1)\)^2\ge0\,,
\]
that is, if $p\le2^\#$ when $d>1$, or $p>1$ when $d=1$. Altogether we have the identity
\begin{multline*}
-\,\frac12\,\frac d{dt}\ix{\(|u'|^2\,\nu+\frac d{p-2}\,\big(|u|^2-\overline u^2\big)\)}\\
=\ix{\left|\,u''-\,\frac{d-1}{d+2}\,(p-1)\,\frac{|u'|^2}u\,\right|^2\nu^2}+\,\gamma_1\ix{\frac{|u'|^4}{u^2}\,\nu^2}\,.
\end{multline*}
Hence we have proved the following result.
\begin{prop}\label{propflowbeta1}
For all $p\in[1,2)\cup(2,2^\#]$ if $d>1$, and for all $p\in[1,2)\cup(+\infty)$ if $d=1$, there exists a constant $\gamma_1\ge0$, such that if $u$ is a positive solution to~\eqref{Eqn:Linear}, then
\[
\frac d{dt}\ix{u^{p}}=0\,,
\]
\[
\frac d{dt}\ix{\(|u'|^2\,\nu+\frac d{p-2}\,\(u^2-\overline u^2\)\)}\le-\,2\,\gamma_1\ix{\frac{|u'|^4}{u^2}\,\nu^2}\,,
\]
where $\gamma_1$ is given by
\[\label{Eqn:gamma1}
\gamma_1=\(\frac{d-1}{d+2}\)^2\,(p-1)\,(2^\#-p)\quad\mbox{if}\quad d>1\,,\quad\gamma_1=\frac{p-1}3\quad\mbox{if}\quad d=1\,.
\]
\end{prop}
This result can be found in~\cite{MR808640}.

\medskip\noindent$\bullet$ \emph{The nonlinear diffusion approach.} Now let us turn our attention towards the nonlinear flow defined by~\eqref{FDE} with $m$ and $\beta$ related by~\eqref{m}, $\kappa=\beta\,(p-2)+1$, and
\[
\rho(t,x)=w^{\beta p}\(\frac{\kappa\,t}{\beta\,p},x\)\,,\quad(t,x)\in\R^+\times[-1,1]\,.
\]
Then the function $w$ satisfies
\be{Eqn:Nonlinear}
\frac{\partial w}{\partial t}=w^{2-2\beta}\(\L w+\kappa\,\frac{|w'|^2}w\)
\ee
and notice that
\[
\frac d{dt}\ix{w^{\beta p}}=\beta\,p\,(\kappa-\beta\,(p-2)-1)\ix{w^{\beta(p-2)}\,|w'|^2\,\nu}\,,
\]
so that $\overline w^{\beta p}=\ix{w^{\beta p}}$ is preserved if $\kappa=\beta\,(p-2)+1$. Recall that~\eqref{Eqn:Nonlinear} is such that $\rho(t,x)=|u(t,x)|^p=|w(t,x)|^{\beta p}$ obeys to the nonlinear flow~\eqref{FDE}. Similarly as in the linear case we calculate:
\begin{multline*}
-\frac1{2\,\beta^2}\,\frac d{dt}\ix{\(\big|(w^\beta)'\big|^2\,\nu+\frac d{p-2}\,\(w^{2\beta}-\overline w^{2\beta}\)\)}\\
=\ix{|w''|^2\,\nu^2}-2\,\frac{d-1}{d+2}\,(\kappa+\beta-1)\ix{w''\,\frac{|w'|^2}w\,\nu^2}\\
+\left[\kappa\,(\beta-1)+\,\frac d{d+2}\,(\kappa+\beta-1)\right]\ix{\frac{|w'|^4}{w^2}\,\nu^2}\,.
\end{multline*}
The r.h.s.~is nonnegative if there exists a $\beta\in\R$ such that
\begin{multline*}
\gamma(\beta):=\kappa\,(\beta-1)+\,\frac d{d+2}\,(\kappa+\beta-1)-\(\frac{d-1}{d+2}\,(\kappa+\beta-1)\)^2\\
=\big(1+\beta\,(p-2)\big)\,(\beta-1)+\frac d{d+2}\,\beta\,(p-1)-\(\frac{d-1}{d+2}\,\beta\,(p-1)\)^2\ge0\,.
\end{multline*}
With the choice of $\beta$ as in Proposition~\ref{prop:SufficientConsidtion},~$\gamma$ is nonnegative. Indeed we have
\[\label{Eqn:gamma}
\gamma(\beta)=-\,1+2\,\mathsf b\,\beta-\,\mathsf a\,\beta^2
\]
with $\mathsf a=\frac{(d-1)^2\,p^2-3\,(d^2+2)\,p+3\,(d^2+2\,d+3)}{(d+2)^2}$ and $\mathsf b=\frac{d+3-p}{d+2}$ and the reduced discriminant
\[
\mathsf b^2-\mathsf a=\frac{4\,d}{(d+2)^2}\,(p-1)\,\big(2\,d-\,p\,(d-2)\big)
\]
takes nonnegative values when $d\ge3$ if $1\le p\le2^*$. The equation $\gamma(\beta)=0$ has at most two solutions $\beta=\beta_\pm(p,d)$, which are the two roots of the polynomial $\beta\mapsto\gamma(\beta)$ given in Section~\ref{Sec:ultraspherical}. Notice here that when $p=2^*$ and $d\ge4$, $\gamma(\beta)=0$ has a single root $\beta=\beta_\pm(2^*,d)=(d-2)/(d-3)$ and that~\eqref{alpha} follows from our computations. The case $d=3$ and $p=6$ is a limit case in~\eqref{Eqn:Nonlinear} corresponding to $\beta\to+\infty$ and can be dealt with directly using~\eqref{FDE}. In dimension $d=2$ and $1$, the discriminant is respectively $2\,(p-1)$ and $\frac49\,(p-1)(p+2)$ and takes nonnegative values for any $p\ge1$. Altogether we obtain the identity
\begin{multline*}
-\frac1{2\,\beta^2}\,\frac d{dt}\ix{\(\big|(w^\beta)'\big|^2\,\nu+\frac d{p-2}\,\(w^{2\beta}-\overline w^{2\beta}\)\)}\\
=\ix{\left|\,w''-\,\frac{d-1}{d+2}\,(\kappa+\beta-1)\,\frac{|w'|^2}w\,\right|^2\nu^2}+\gamma(\beta)\ix{\frac{|w'|^4}{w^2}\,\nu^2}\,.
\end{multline*}
Notice that $\gamma(1)=\gamma_1$, so that the above identity generalizes the computation done for the heat flow. We have proved the following result.
\begin{prop}\label{propflowbetanot1} For all $p\in[1, 2)\cup (2,2^*]$ if $d\ge4$, for all $p\in[1, 2)\cup (2,2^*)$ if $d=3$, and for all $p\in[1,2)\cup(+\infty)$ if $d=1$ or $d=2$, there exist two constants, $\beta\in\R$ and $\gamma>0$, such that if $w$ is a solution to~\eqref{Eqn:Nonlinear}, then
\[
\frac d{dt}\ix{w^{\beta p}}=0
\]
and
\begin{multline*}
\frac d{dt}\ix{\(\big|(w^\beta)'\big|^2\,\nu+\frac d{p-2}\,\(w^{2\beta}-\overline w^{2\beta}\)\)}\\
\le-\,2\,\beta^2\,\gamma(\beta)\ix{\frac{|w'|^4}{w^2}\,\nu^2}\,.
\end{multline*}
\end{prop}
This result can be found in~\cite{DEKL2014}.
\begin{figure}[ht]
\includegraphics{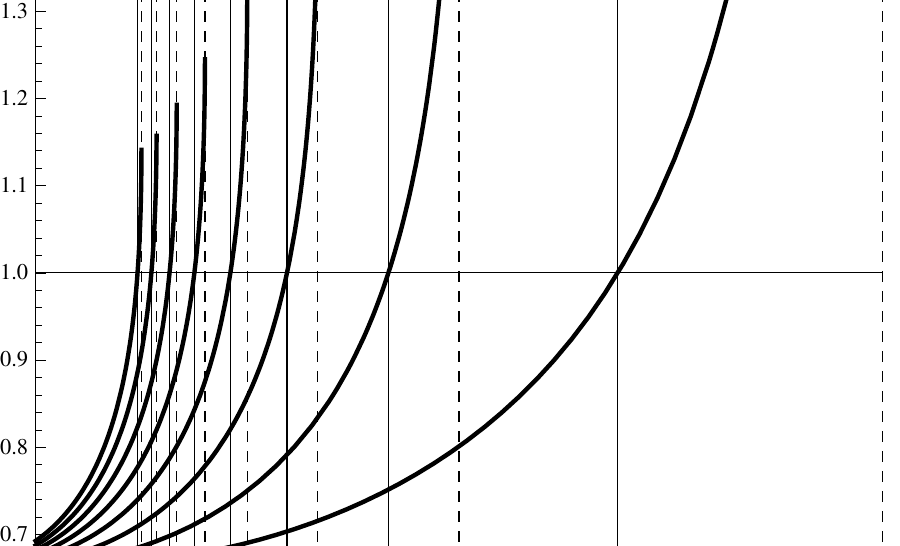}\caption{\label{F3pdf}\small\sl The function $p\mapsto\beta_-(p,d)$ for various values of $d=3$, $4$,... $10$. The straight horizontal line corresponds to $\beta=1$. The plain straight vertical lines correspond to $p=2^\#$ and the dashed straight vertical lines correspond to $p=2^*$.}
\end{figure}

\medskip\noindent$\bullet$ \emph{Proof of Proposition~\ref{prop:obstruction2}.} We are now in a position to build our counter-example, which is the second obstruction we search for.

With $\alpha=\frac{d-1}{d+2}\,\beta\,(p-1)$, $\beta=\beta_-(p,d)$ and $w$ such that
\[
w(z)= (\mathsf a+\mathsf b\,z)^\frac1{1-\alpha}
\]
for some positive constants $\mathsf a$ and $\mathsf b$, we observe that
\be{Eqn:w}
w''=\alpha\,\frac{|w'|^2}w\,.
\ee
Next we consider $u=w^\beta$ and compute
\begin{eqnarray*}
&&\frac1\beta\,w^{1-\beta}\,u''=w''+(\beta-1)\,\frac{|w'|^2}w=(\alpha+\beta-1)\,\frac{|w'|^2}w\,,\\
&&\frac1\beta\,w^{1-\beta}\,\frac{|u'|^2}u=\beta\,\frac{|w'|^2}w\,.
\end{eqnarray*}
If we take this function $w$ as initial datum and consider the flow defined by~\eqref{heat} and~\eqref{Eqn:Linear}, then with $\rho(t,x)=|w(t,x)|^{\beta p}$ we find that
\begin{multline*}
-\frac d{dt}\mathsf F[\rho]_{|t=0}=-\,\mathrm D\mathsf F[\rho]\cdot\L\rho\\
\hspace*{3cm}=\ix{|u''|^2\,\nu^2}-2\,\frac{d-1}{d+2}\,(p-1)\ix{u''\,\frac{|u'|^2}u\,\nu^2}\\
\hspace*{3cm}+\frac d{d+2}\,(p-1)\ix{\frac{|u'|^4}{u^2}\,\nu^2}\\
=-\,\mathsf A\ix{\frac{|u'|^4}{u^2}\,\nu^2}
\end{multline*}
where
\[
-\,\mathsf A=(\alpha+\beta-1)^2-2\,\frac{d-1}{d+2}\,(p-1)\,(\alpha+\beta-1)\,\beta+\frac d{d+2}\,(p-1)\,\beta^2\,.
\]
After eliminating $\alpha$ and $\beta$, we can observe that $p\mapsto \mathsf A(p)$ is positive. An algebraic proof is given below, in Lemma~\ref{Lem:algebraic}. This concludes the proof of Proposition~\ref{prop:obstruction2}.\qed

\begin{figure}[ht]
\includegraphics[width=5.5cm]{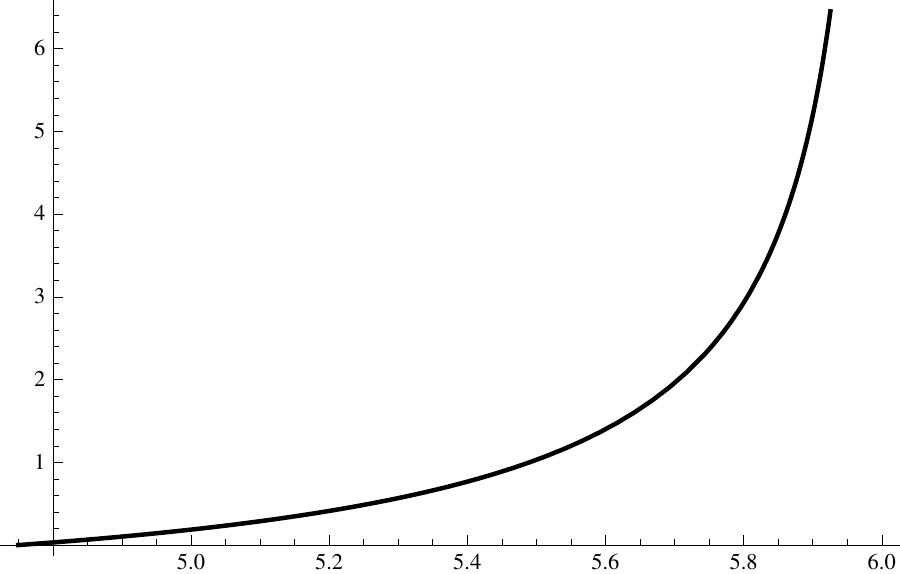}\hspace*{9pt}\includegraphics[width=5.5cm]{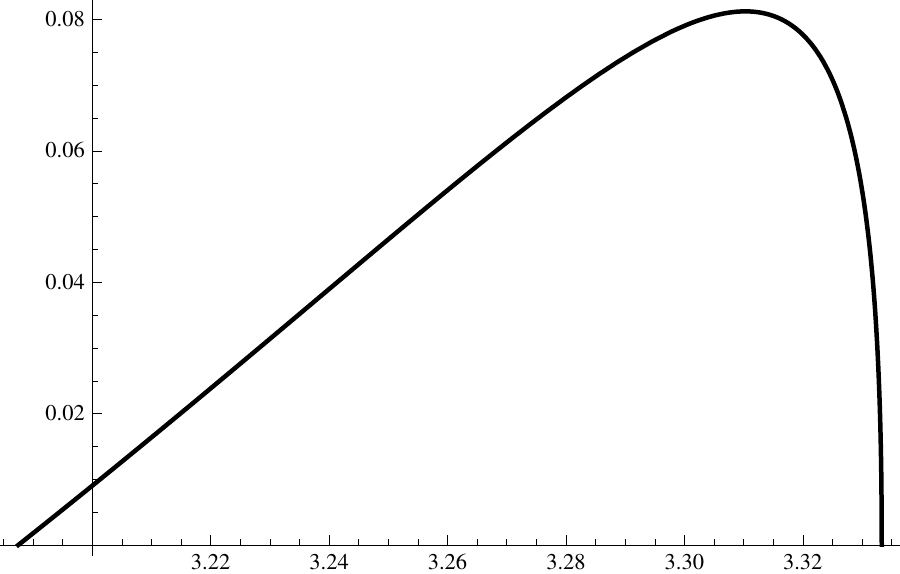}
\caption{\label{F4}\small\sl The function $p\mapsto\mathsf A(p)$ for various values of $d=3$ (left) and $d=5$ (right). The patterns are similar for all $d\ge4$. Here $p$ is in the range $2^\sharp<p<2^*$.}
\end{figure}

\begin{lemma}\label{Lem:algebraic} Assume that $d\ge3$. With the above notations, $\mathsf A$ is positive when $p\in(2^\sharp,2^*)$ and $\beta=\beta_-(p,d)$.\end{lemma}
\begin{proof} With $\alpha=\frac{d-1}{d+2}\,\beta\,(p-1)$, we get that
\[
\mathsf A=\frac{(d-1)^2\,p^2-\,(3\,d^2-\,2\,d+2)\,p+d^2-\,4\,d-\,3}{(d+2)^2}\,\beta^2+\,2\,\beta-1
\]
and the equation $\mathsf A=0$ has at most two solutions $\beta=B_\pm(p,d)$ with
\[
B_\pm(p,d):=\frac{d+2}{d+2\pm\,(d-1)\sqrt{(p-1)\,(p-2^\#)}}\,.
\]
Elementary computations show that $\frac1{B_-}<\frac1{\beta_-}<\frac1{B_+}$ if $p\in(2^\#,2^*)$. Indeed, it is elementary to check that
\[
\pm\,\frac{d+2}{\sqrt{p-1}}\kern -1pt\(\frac1{B_\pm}-\frac1{\beta_-}\)\kern -1pt=\pm\sqrt{p-1}\mp\sqrt{d\,(d-2)}\,\sqrt{2^*-p}+\,(d-1)\,\sqrt{p-2^\#}
\]
is positive if $p\in(2^\#,2^*)$. Moreover, we have that $\mathsf A$ is positive for any $p\in(2^\#,2^*)$ if $B_+<\beta<B_-$, which concludes the proof.\end{proof}

\begin{figure}[ht]
\includegraphics[width=6cm]{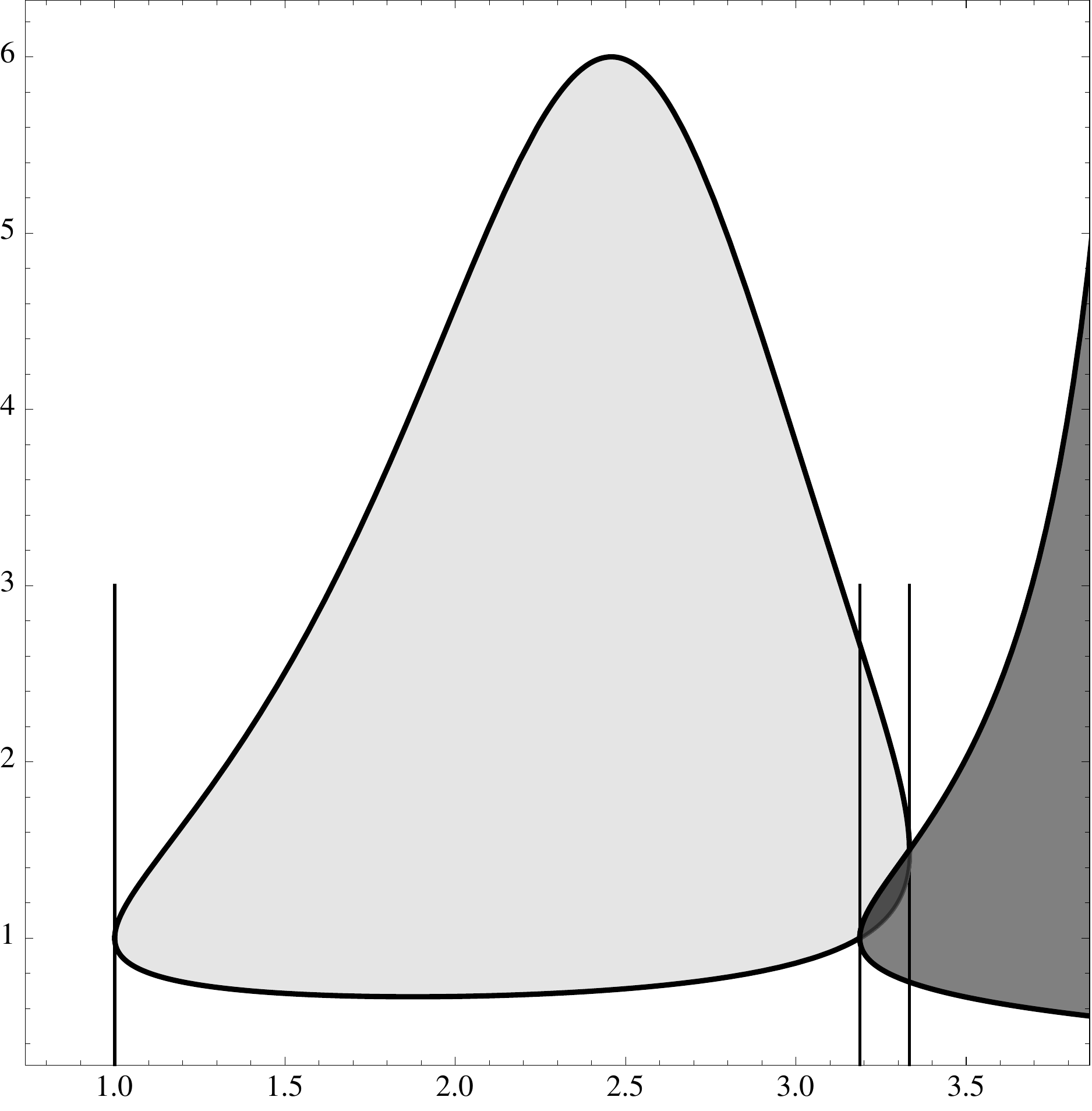}
\caption{\label{F5}\small\sl Here is the $(p,\beta)$ representation when $d=5$. The grey area and the overlapping region (dark grey area) correspond to $\mathsf A>0$. There is a region in which $\mathsf A$ is positive, which intersects the admissible range (light grey area) of $\beta$. Vertical lines are located at $p=1$, $p=2^\#$ and $p=2^*$.}
\end{figure}

\begin{figure}[ht]\bigskip
\includegraphics[width=5.5cm]{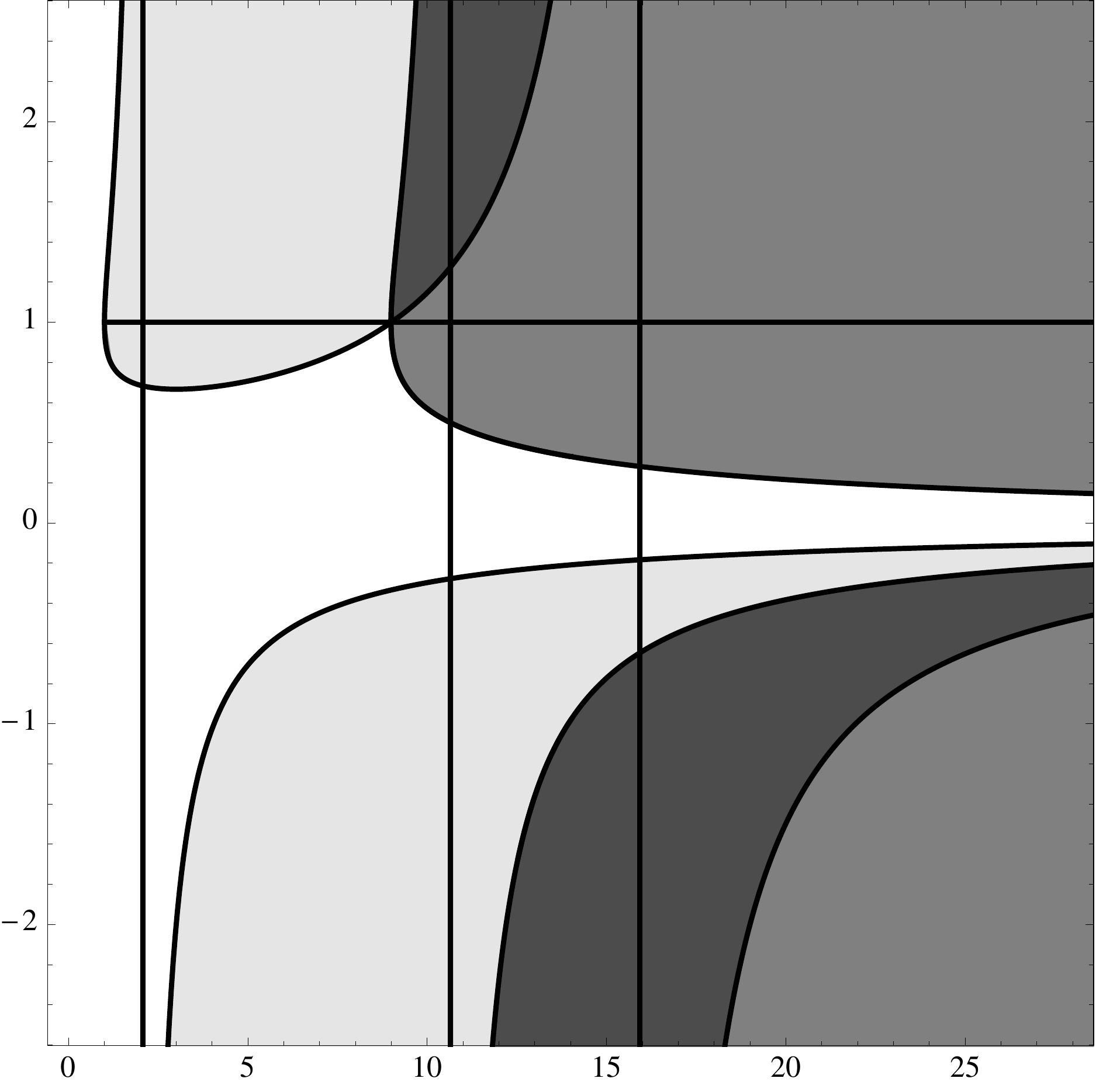}\hspace*{9pt}\includegraphics[width=5.5cm]{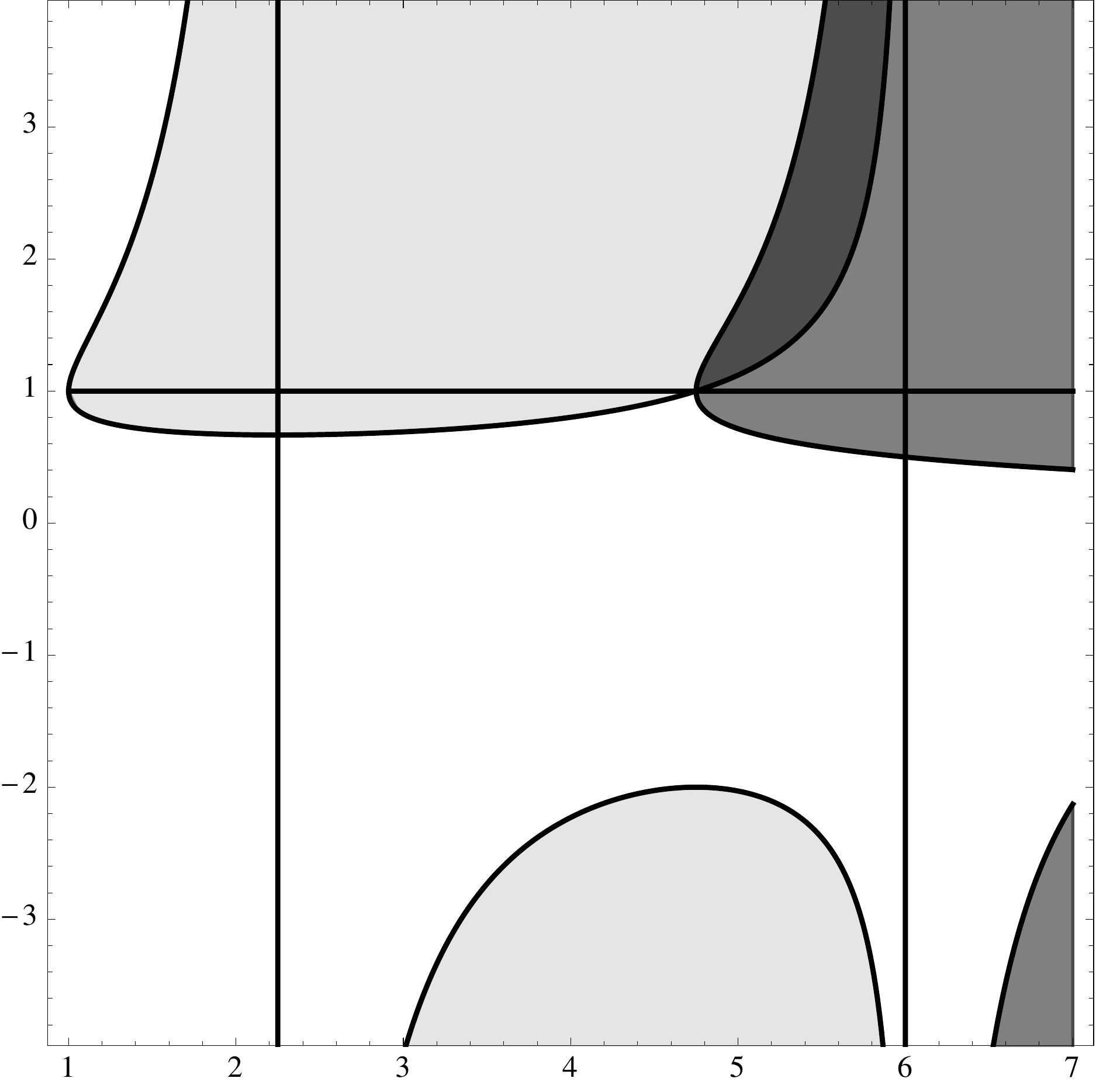}
\caption{\label{F6}\small\sl The discussion of the admissible range of $\beta$ and the positivity region of $\mathsf A$ is more complicated in dimension $d=3$ (right) than for $d\ge4$. A similar discussion can also be done in dimension $d=2$ (left). Again the light grey areas correspond to admissibility of $\beta$, the grey areas to the zone where $A>0$, and the dark grey areas to the zones which are interesting to us, where $\beta$ is admissible, and $\mathsf A>0$.}
\end{figure}

\clearpage

\section{Improvements}\label{Sec:Improvements}

In this last section we investigate improved inequalities or, to be precise, improvements on the optimal constants, that can be achieved in inequalities~\eqref{Ineq:GNS} and~\eqref{Ineq:logSob} when additional integral constraints are imposed. The general message is that improvements can always be obtained, but semi-explicit (and probably non-optimal) constants are known only when $p<2^\#$. The main goal of this section is to sketch the proof of Theorem~\ref{Thm:Improvement}. Let us start by reviewing a few results.

\medskip If $p\in[1,2)$, we may refer to~\cite[Theorem~1.2]{DEKL2014} for an improvement based on the spectral decomposition associated with the Laplace-Beltrami operator on $\S^d$ and standard $\mathrm L^2(\S^d)$ orthogonality constraints.

A more striking improvement has been obtained in~\cite[Sections~4.5 and~4.6]{DEKL2012}. Under the assumption that $f(-z)=f(z)$ for any $z\in(-1,1)$ a.e., Inequality~\eqref{GNS-ultra} can be improved to
\[
-\scal f{\L f}=\ix{|f'|^2\;\nu}\ge\frac{d^2+(d-1)^2\,(2^\#-p)}{d\,(p-2)}\,\(\nrmx fp^2-\nrmx f2^2\)
\]
for any $p\in[1,2)\cup(2,2^\#)$. When $p=1$, we observe that $[d^2+(d+1)^2\,(2^\#-p)]/d=2\,(d+1)$ is the second positive eigenvalue of $-\,\L$. As a limit case corresponding to $p=2$, the improvement also covers the case of the logarithmic Sobolev inequality and shows that
\[
-\scal f{\L f}\ge\frac{d^2+4\,d-1}{2\,d}\ix{|f|^2\,\log\(\frac{|f|^2}{\nrmx f2^2}\)}\,,
\]
for any function $f\in\mathrm H^1((-1,1),d\nu_d)$ such that $f(-z)=f(z)$ for any $z\in(-1,1)$ a.e. We will state a better result (Theorem~\ref{Thm:Antipodal}) under an \emph{antipodal symmetry} assumption at the end of this section.

\medskip Let us state a first new result on improvements that provides us with a non-constructive constant.
\begin{prop}\label{Prop:BE} Assume that $p\in[1,2)\cup(2,2^*]$ and $q\in(2,2^*)$. Then there exists a constant $\mathsf C_{p,q}>d$ such that
\[
-\scal f{\L f}=\ix{|f'|^2\;\nu}\ge\frac{\mathsf C_{p,q}}{p-2}\,\(\nrmx fp^2-\nrmx f2^2\)
\]
for any a.e.~function $f\in\mathrm H^1((-1,1),d\nu_d)$ such that
\[
\ix{z\,|f|^q}=0\,.
\]
\end{prop}
This result is based on a Bianchi-Egnell type improvement in the subcritical range and generalizes the result of~\cite{MR1124290} to $p<2^*$.
\begin{proof} A simple spectral decomposition shows that $\mathsf C_{p,1}=2\,(d+1)$.

Assume next that $p\in(1,2)\cup(2,2^*)$. It has been proved in~\cite{MR2381156} and in~\cite[Theorem~1.1]{DEKL2014} that there exists a strictly convex function $\Phi$ on $\R^+$ such that $\Phi(0)=0$, $\Phi'(0)=1$ and
\begin{multline*}
\nrm{\nabla u}2^2\ge d\,\Phi\(\frac{\nrm up^2-1}{p-2}\)\quad\forall\,u\in\mathrm H^1(\S^d,d\mu)\\
\mbox{such that}\quad\nrm u2^2=1\,.
\end{multline*}
The same improvement is also true in the context of the ultraspherical operator as can be checked from~\cite[Sections~3 and~4]{DEKL2014}. Hence we have that
\begin{multline*}
-\scal f{\L f}\ge d\,\Phi\(\frac{\nrmx fp^2-1}{p-2}\)\quad\forall\,f\in\mathrm H^1((-1,1),d\nu_d)\\
\mbox{such that}\quad\nrmx f2^2=1\,.
\end{multline*}
It is clear that the infimum of $-(p-2)\,\scal f{\L f}/\(\nrmx fp^2-\nrmx f2^2\)$ can be taken under the additional constraint $\nrmx f2=1$ without restriction and that it can be achieved only in the limit as $f\to1$. If the limit is equal to $d$, then $f-1$ is up, to higher order terms, proportional to $z$, which contradicts the constraint $\ix{z\,|f|^q}=0$. This proves that $\mathsf C_{p,q}>d$.

If $p=2^*$, Inequality~\eqref{GNS-ultra} is equivalent to the classical Sobolev inequality on $\R^d$, as can be shown using the stereographic projection. Arguing by contradiction, as in~\cite{MR1124290}, and using the fact that, due to the constraint, the function (after stereographic projection) is asymptotically in the orthogonal to the manifold of Aubin-Talenti functions, we get that $\mathsf C_{2^*\!,q}>d$. Of course one has to take care of all invariances as was done in~\cite{MR1124290}, that is, of the conformal invariance on $\S^d$. Technical details are left to the reader.\end{proof}

Now let us turn our attention to the proof of Theorem~\ref{Thm:Improvement}. Inequality~\eqref{Ineq:GNSimproved} follows from Proposition~\eqref{Prop:BE} when $p=q<2^*$ and $u$ depends only on $z\in(-1,1)$. For simplicity, we will argue in this simplified setting and only indicate how to extend the result to the general case. In analogy with the definition of $\Lambda^\star$, let us define
\[
\lambda^\star\quad:=\inf_{\begin{array}{c}v\in\mathrm H^1_+((-1,1),d\nu_d)\\ \ix{v}=1\\ \ix{z\,|v|^p}=0\end{array}}\frac{\ix{{(\L v)^2}}}{\ix{|v'|^2\;\nu}}>d
\]
and consider the inequality
\begin{multline}\label{Ineq:GNSimprovedSimplified}
\ix{|f'|^2\;\nu}+\frac\lambda{p-2}\,\nrmx f2^2\ge\frac\lambda{p-2}\,\nrmx fp^2\\
\forall\,f\in\mathrm H^1((-1,1),d\nu_d)\;\mbox{s.t.}\;\ix{z\,|f|^p}=0
\end{multline}
\begin{prop}\label{Prop:ImprovementSimplfied} For any $p\in(2,2^\#)$, inequality~\eqref{Ineq:GNSimprovedSimplified}
holds with
\[
\lambda\ge d+\frac{(d-1)^2}{d\,(d+2)}\,(2^\#-p)\,(\lambda^\star-d)\,.
\]\end{prop}
\begin{proof} We consider the heat flow~\eqref{heat} applied to a function $\rho$ with initial datum $|f|^p$, or equivalently, the flow defined by~\eqref{Eqn:Linear} applied to a function~$u$ with initial datum $f$. We observe that
\begin{multline*}
\frac d{dt}\ix{z\,|u|^p}=p\ix{z\,|u|^{p-2}\,u\(\L u+(p-1)\,\frac{|u'|^2}u\)}\\=-\,p\ix{|u|^{p-2}\,u\,u'\nu}=-\,d\ix{z\,|u|^p}\,.
\end{multline*}
Hence, if $\ix{z\,|u|^p}=0$ at $t=0$, this is also true for any $t>0$. From now on, we shall assume that this constraint is satisfied.

With no restriction, we may assume that $u$ is positive. Instead of writing that
\begin{multline*}
-\,\frac12\,\frac d{dt}\ix{\(|u'|^2\,\nu+\frac d{p-2}\,|u|^2\)}\\
=\ix{|u''|^2\,\nu^2}-2\,\frac{d-1}{d+2}\,(p-1)\ix{u''\,\frac{|u'|^2}u\,\nu^2}\\
+\frac d{d+2}\,(p-1)\ix{\frac{|u'|^4}{u^2}\,\nu^2}\,,
\end{multline*}
we can write that
\begin{multline*}
-\,\frac12\,\frac d{dt}\ix{\(|u'|^2\,\nu+\frac d{p-2}\,|u|^2\)}\\
=\(1-(p-1)\,\frac{(d-1)^2}{d\,(d+2)}\)\ix{|u''|^2\,\nu^2}\\
+\frac d{d+2}\,(p-1)\ix{\left|\frac{|u'|^2}u-\frac{d-1}d\,u''\right|^2\,\nu^2}
\end{multline*}
and observe that $1-(p-1)\,\frac{(d-1)^2}{d\,(d+2)}$ is positive since $p<2^\#:=\frac{2\,d^2+1}{(d-1)^2}$. Using the formula
\be{Eqn:u''}
\ix{|u''|^2\;\nu^2}=\ix{(\L u)^2}-d\ix{|u'|^2\;\nu}\,,
\ee
and the definition of $\lambda^\star$, we find that
\[
-\,\frac12\,\frac d{dt}\ix{\(|u'|^2\,\nu+\frac\lambda{p-2}\,|u|^2\)}\ge0
\]
if
\[
\lambda=d+\frac{(d-1)^2}{d\,(d+2)}\,(2^\#-p)\,(\lambda^\star-d)\,.
\]
Using $u=1+\varepsilon\,u_2$ with $\varepsilon<1/2$ and $u_2(z)=z^2-2$ as a test function, we obtain that $\lambda^\star\le2\,(d+1)$ since $-\,\L u_2=2\,(d+1)\,u_2$.
\end{proof}

Proposition~\ref{Prop:ImprovementSimplfied} provides an improvement of the constant~$\lambda$ because of the following estimate.
\begin{prop}\label{Prop:>d} For any $p\in(2,2^*)$, we have that
\[
\lambda^\star>d\,.
\]
\end{prop}
Notice that the estimate of $\lambda$ based on $\lambda^\star$ is a constructive but non explicit estimate, as we do not know the value of $\lambda^\star$, and also that
\[
\frac{\ix{{(\L u)^2}}}{\ix{|u'|^2\;\nu}}=d\quad\mbox{and}\quad\ix{z\,|u|^p}=0
\]
if $u(z)=z$. However the condition $u>0$ on $(-1,1)$ is not satisfied in this example. Hence the positivity of $u$ in the infimum is crucial.

\begin{proof}[Proof of Proposition~\ref{Prop:>d}] For any $\mu>0$, by expanding the square in
\[
\ix{|\L u+\mu\,(u-\bar u)|^2}\ge0\,,
\]
and after an integration by parts, we observe that
\begin{multline*}
\ix{{(\L u)^2}}-\mu\ix{|u'|^2\;\nu}\\
\ge\mu\(\ix{|u'|^2\;\nu}-\mu\ix{|u-\bar u|^2}\)\,.
\end{multline*}
As a consequence, the infimum $\lambda^\star$ can be estimated by
\[
\lambda^\star\ge\inf_{\begin{array}{c}u\in\mathrm H^1_+((-1,1),d\nu_d)\\ \ix u=1\\ \ix{z\,|u|^p}=0\end{array}}\frac{\ix{|u'|^2\;\nu}}{\ix{{|u-1|^2}}}\,,
\]
which is achieved by some nonnegative function $u$. The Maximum Principle applies and shows that the minimizer is then positive. Since the optimality condition $\ix{(\L u)^2}=d\ix{|u'|^2\;\nu}$ would imply that $u$ is of the form $u(z)=a+b\,z$ for some $a$ and $b\in\R$ such that $|b|<|a|$, it is clear that the constraint $\ix{z\,u^p}=0$ cannot be matched unless $a=0$, which violates the positivity of $u$. This proves that $\lambda^\star=d$ is impossible.\end{proof}

Theorem~\ref{Thm:Improvement} can be proved using the same strategy as for Propositions~\ref{Prop:ImprovementSimplfied}--\ref{Prop:>d}, except that the flow~\eqref{heat} associated with the ultraspherical operator has to be replaced by the heat flow on $\S^d$ given by~\eqref{Eqn:Heat}. Computations are more technical and can be found in \cite{1302}. The key observation is again that $\iS{x\,\rho(t,x)}=0$ for any $t\ge0$ if $\iS{x\,\rho(t=0,x)}=0$.

The estimates of Theorem~\ref{Thm:Improvement} and Propositions~\ref{Prop:ImprovementSimplfied}--\ref{Prop:>d} are constructive for any $p\in(2,2^\#)$, but the values of the constants $\lambda^\star$ and $\Lambda^\star$ are not known so far. From their definitions we know that $\lambda^\star\ge\Lambda^\star$ but it is an open question so far to decide if equality holds or not.

\medskip In the limit case $p=2$, one can get the explicit estimate
\[
\Lambda^\star\ge d+\frac{2\,(d+2)}{2\,(d+3)+\sqrt{2\,(d+3)\,(2\,d+3)}}\,. 
\]
As a consequence, we obtain the following result.
\begin{prop}\label{ImprovedLogSob} Let $d\ge2$. For any $u\in\mathrm H^1(\S^d,d\mu)\setminus\{0\}$ such that $\iS{x\,|u|^2}=0$, we have
\[
\iS{|\nabla u|^2}\ge\,\frac\delta2\iS{|u|^2\,\log\(\frac{|u|^2}{\nrm u2^2}\)}
\]
with $\delta:=d+\frac 2d\,\frac{4\,d-1}{2\,(d+3)+\sqrt{2\,(d+3)\,(2\,d+3)}}$.
\end{prop}
\begin{proof} Our proof relies on an estimate of $\Lambda^\star$ when $p=2$. We write $u=1+a\cdot x+v$ where $v$ is orthogonal to the constants and all the $x_i$ with $i=1,\ldots,d+1$ and $a\in\mathbb R^{d+1}$. Moreover $u$ has to satisfy the constraint $u\ge 0$. Hence we have to minimize $\mathsf E=\mathsf E[v]$ such that
\[
\mathsf E=\frac{d\,b+\iS{|\nabla v|^2}}{b+\iS{v^2}}\ge\mathsf e\(\iS{v^2}\)
\]
where
\[
b:=\iS{(a\cdot x)^2}=\frac{|a|^2}{d+1}\,,\quad c:=\iS{v^2}
\]
and
\[
\mathsf e(c):=\frac{d\,b+2\,(d+1)\,c}{b+c}=2\,(d+1)-\frac{(d+2)\,b}{b+c}
\]
is monotone increasing in $c$. Our strategy is to bound $c$ from below in terms of $b$ and then minimize the resulting expression in terms of $b\ge0$.

\smallskip\noindent\emph{ First estimate.} From the fact that $u=1+a\cdot x+v\ge0$ we get that $-\,a\cdot x\le(1+v(x))$, \emph{i.e.},
\[
a\cdot x\ge-\,(1+v(x))\,.
\]
By exchanging $x$ with $-x$ we also get $a\cdot x\le1+v(-x)$. Hence we have that
\[
1+v(-x)\ge a\cdot x\ge-\,(1+v(x))\,,
\]
\emph{i.e.},
\[
|a\cdot x|\le\max\left\{|1+v(-x)|,|1+v(x)|\right\}
\]
or
\[
|a\cdot x|^2\le\max\Big\{|1+v(-x)|^2,|1+v(x)|^2\Big\}\le|1+v(-x)|^2+|1+v(x)|^2
\]
and now integrate. This proves that
\[
b=\iS{(a\cdot x)^2}\le2\(1+\iS{v^2}\)\,.
\]
We get a first inequality
\[
\iS{v^2}\ge\frac b2-1\,.
\]
This establishes the estimate
\be{FirstEstimate_p=2}
\mathsf E\ge\mathsf e\(\frac b2-1\)=2\,(d+1)-\frac{2\,(d+2)\,b}{3\,b-2}
\ee
where the r.h.s.~is an increasing function of $b>2$.

\smallskip\noindent\emph{ Second estimate.} We write $\iS{(1+a\cdot x+v)^2\,x_i}=0$ as
\[
\iS{v\,(v+2\,a\cdot x)\,x_i}=-\frac{2\,a_i}{d+1}\,.
\]
Note that $v$ is perpendicular to $x_i$. By Schwarz and then summing over $i$ we get
\[
\frac{4\,b}{d+1}=\frac{4\,|a|^2}{(d+1)^2}\le\iS{v^2}\(\iS{v^2}+\frac{4\,|a|^2}{d+1}\)=c\,(c+4\,b)\,.
\]
Setting $b=\frac{|a|^2}{d+1}$ as above, one easily gets a second inequality
\[
\iS{v^2}\ge2\,\sqrt{b^2+\frac b{d+1}}-2\,b=\frac2{d+1}\,\frac b{b+\sqrt{b^2+\frac b{d+1}}}\,.
\]
Hence we have found that
\be{SecondEstimate_p=2}
\mathsf E\ge\mathsf e\(2\,\sqrt{b^2+\frac b{d+1}}-2\,b\)=2\,(d+1)-\frac{(d+1)\,(d+2)}{d+1+\frac2{b+\sqrt{b^2+\frac b{d+1}}}}\,.
\ee
In this second estimate the r.h.s.~as a function of $b$ is monotone decreasing.

\smallskip\noindent\emph{ Conclusion of the proof.} By combining~\eqref{FirstEstimate_p=2} and~\eqref{SecondEstimate_p=2}, we obtain a global estimate which is independent of $b\ge0$. Let us solve
\[
\mathsf e\(\frac b2-1\)=\mathsf e\(2\,\sqrt{b^2+\frac b{d+1}}-2\,b\)\,,\quad b>2\,.
\]
All computations done, this gives
\begin{multline*}
b=b_\star(d)=\frac29\,\frac{2\,\sqrt{2\,(d+3)\,(2\,d+3)}+5\,d+9}{d+1}\\
\mbox{and} \quad \mathsf E\ge\mathsf e\(\frac{b_\star(d)}2-1\)=d+\frac{2\,(d+2)}{2\,(d+3) +\sqrt{2\,(d+3)\,(2\,d+3)}}\,.
\end{multline*}
Notice that $d\mapsto b_\star(d)$ is decreasing with $\lim_{d\to\infty}b_\star(d)=2$.

We have shown that
\[
\Lambda^\star\ge d+\frac{2\,(d+2)}{2\,(d+3)+\sqrt{2\,(d+3)\,(2\,d+3)}}\,.
\]
Conclusion holds for the same reasons as in Proposition~\ref{Prop:ImprovementSimplfied}.
\end{proof}

To conclude this section, let us state a last improvement that can be obtained under a stronger constraint. With the additional restriction of \emph{antipodal symmetry}, that is
\be{Eqn:antipodal}
u(-x)=u(x)\quad\forall\,x\in\S^d\,,
\ee
we can state an explicit result that goes as follows.
\begin{prop}\label{Prop:Antipodal} If $p\in(1,2)\cup(2,2^\#)$, we have
\[
\nrm{\nabla u}2^2\ge\frac{d^2+(d-1)^2\,(2^\#-p)}{d\,(p-2)}\(\nrm up^2-\nrm u2^2\)
\]
for any $u\in\mathrm H^1(\S^d,d\mu)$ such that~\eqref{Eqn:antipodal} holds. The limit case $p=2$ corresponds to the improved logarithmic Sobolev inequality
\[
\nrm{\nabla u}2^2\ge\frac{d^2+4\,d-1}{2\,d}\iS{|u|^2\,\log\(\frac{|u|^2}{\nrm u2^2}\)}
\]
for any $u\in\mathrm H^1(\S^d,d\mu)\setminus\{0\}$ such that~\eqref{Eqn:antipodal} holds. \end{prop}
See \cite[Section~4.5]{DEKL2012} for a proof based on the ultraspherical operator. It is easily recovered by taking the formula in Proposition~\ref{Prop:ImprovementSimplfied} and replacing $\lambda^\star$ by the second positive eigenvalue of the ultraspherical operator, namely $\lambda_2=2\,(d+1)$. As usual the case of the logarithmic Sobolev inequality is obtained by taking the limit as $p\to2$. This result is based on the heat flow~\eqref{heat} and one can get a \emph{better} result which also covers the range $p\in[2^\#,2^*]$ using a nonlinear diffusion.
\begin{thm}\label{Thm:Antipodal} If $p\in(1,2)\cup(2,2^*)$, we have
\[
\iS{|\nabla u|^2}\ge\frac d{p-2}\left[1+\frac{(d^2-4)\,(2^*-p)}{d\,(d+2)+p-1}\right]\(\nrm up^2-\nrm u2^2\)
\]
for any $u\in\mathrm H^1(\S^d,d\mu)$ such that~\eqref{Eqn:antipodal} holds. The limit case $p=2$ corresponds to the improved logarithmic Sobolev inequality
\[
\iS{|\nabla u|^2}\ge\frac d2\frac{(d+3)^2}{(d+1)^2}\iS{|u|^2\,\log\(\frac{|u|^2}{\nrm u2^2}\)}
\]
for any $u\in\mathrm H^1(\S^d,d\mu)\setminus\{0\}$ such that~\eqref{Eqn:antipodal} holds. \end{thm}
The above constants are probably not optimal as we get no improvement for $p=2^*$ while one is expected because of the result in~\cite{MR1124290}. We may also observe that when $d=3$, the quotient of the sphere~$\S^3$ by the antipodal symmetry is homeomorphic to the group of rotations $SO(3)$ on~$\R^3$. The range in $p$ covered in Theorem~\ref{Thm:Antipodal}, that is $p\in[1,2^*]$, is larger than the range covered in Proposition~\ref{Prop:Antipodal}, namely $p\in[1,2^\#]$. Moreover, a tedious but elementary computation shows that
\[
d\left[1+\frac{(d^2-4)\,(2^*-p)}{d\,(d+2)+p-1}\right]-\frac{d^2+(d-1)^2\,(2^\#-p)}d\ge\frac{(d-1)^2\,(p-1)^2}{d\,\big(d\,(d+2)+p-1\big)}
\]
is nonnegative for any $p\in[1,2^\#]$, then showing that the constant in Theorem~\ref{Thm:Antipodal} is larger than the constant in Proposition~\ref{Prop:Antipodal}.

\begin{proof}[Proof of Theorem~\ref{Thm:Antipodal}] The proof is implicitly contained in \cite {1302}. Using the flow
defined by
\[
\frac{\partial w}{\partial t}=w^{2-2\beta}\(\Delta w+\big(1+\beta\,(p-2)\big)\,\frac{|\nabla w|^2}w\)
\]
it was shown that for all $0<\lambda<\lambda_\star$, where
\[
\lambda_\star:=\inf_{\begin{array}{c}w\in\mathrm H^2(\S^{d})\\\nabla w\not\equiv0\end{array}}\frac{\displaystyle\iS{\Big[(1-\theta)\,(\Lap w)^2+\frac{\theta\,d}{d-1}\,{\rm Ric}(\nabla w,\nabla w)\Big]}}{\iS{|\nabla w|^2}}\;,
\]
the equation
\be{Eqn}
-\,\Lap u+\frac\lambda{p-2}\,\(u-u^{p-1}\)=0
\ee
has a unique positive solution $u\in C^2(\S^d)$, which is constant and equal to $1$ for all $p\in(2,2^*)$. Here, ${\rm Ric}(\nabla w,\nabla w)$ denotes the Ricci curvature, which on $\S^d$ is given by $(d-1)\,|\nabla w|^2$. The constants $\theta$ and $\beta$ are chosen to be
\[
\theta=\frac{(d-1)^2\,(p-1)}{d\,(d+2)+p-1}\quad{\rm and}\quad\beta=\frac{d+2}{d+3-p}\,.
\]
Now one observes that the flow preserves functions that have antipodal symmetry and hence these considerations apply in this case as well. On the space of functions with antipodal symmetry one finds that the operator $(\Lap)^2-2\,(d+1)\,\Lap$ is nonnegative which implies the inequality
\[
\iS{(\Lap u)^2}\ge2\,(d+1)\iS{|\nabla u|^2}\,.
\]
In particular this yields that
\[
\lambda_\star=(1-\theta)\,2\,(d+1)+\theta\,d\,,
\]
which proves the theorem. The improved logarithmic Sobolev inequality follows by taking the limit $p\to2$ and is standard. For more details the reader may consult \cite{1302}.\end{proof}

\section*{Concluding remarks and open problems}\label{Sec:Conclusion}

The limiting exponent $2^\#=\frac{2\,d^2+1}{(d-1)^2}$ for the proofs based on the heat flow is not a technical limitation, since monotonicity cannot be ensured for $p\in(2^\#,2^*]$. On the other hand, when $p<2^\#$, it is possible to prove explicit improvements of the inequalities under an additional integral constraint, in the spirit of the Bianchi-Egnell estimate for the critical Sobolev inequality. Explicit estimates of the optimal constants for constrained problems (with integral constraints) when $p>2^\#$ are so far open questions.

All computations have been done for integer values of $d$ only, because of the $d$-dimensional interpretation of the computations in Section~\ref{Sec:Intro}. However, computations of Sections~\ref{Sec:ultraspherical}--\ref{Sec:Improvements} can also be done for non-integer values of $d$. In this paper, computations have been limited to the $d$-dimensional sphere and even to the case of the ultraspherical operator, but the exponent $2^\#$ also appears on general manifolds with positive curvature: see~\cite{1302} for a discussion and some extensions. The discussion of the general case is however less interesting because the equation which generalizes~\eqref{Eqn:w} has, in general, no explicit solution, and also because the constant obtained by the flow method is only a lower bound for the optimal constant in the interpolation inequality. By Obata's theorem, this bound is actually not optimal except when the manifold is a sphere.

It is an open question to understand whether the improved rates that can be obtained in the asymptotic regimes also determine optimal constants in the global interpolation inequalities. The improvements of Section~\ref{Sec:Improvements} show that there are still lots of issues to understand in the case of constrained problems for subcritical and critical interpolation inequalities. It also emphasizes the role of the exponent $p=2^\#$ and its connection with the heat flow.

\bigskip\noindent{\bf Acknowledgements.} The authors thank Dominique Bakry for raising the key question studied in this paper and for the friendly interactions that took place on this occasion. They also thank a referee for a careful reading which was helpful for improving the manuscript.
\par\medskip\noindent{\small\copyright\,2015 by the authors. This paper may be reproduced, in its entirety, for non-commercial purposes.}


\newpage
\end{document}